\mag=\magstep1  

\documentclass{article}
\usepackage{amsmath,amssymb,theorem}
\usepackage{color}

\theorembodyfont{\upshape}

 \makeatletter
 \@addtoreset{equation}{section}
 \makeatother

 \newtheorem{ittheorem}{Theorem}
 \newtheorem{itlemma}{Lemma}
 \newtheorem{itproposition}{Proposition}
 \newtheorem{itdefinition}{Definition}

 \newtheorem{itremark}{Remark}
 \newtheorem{itclaim}{Claim}
 \newtheorem{itcorollary}{\bf Corollary}
\def\st{\widetilde\sigma}
\def\Ga{\Gamma}

 \newenvironment{theorem}{\addtocounter{equation}{1}
 \begin{ittheorem}}{\end{ittheorem}}

 \newenvironment{lemma}{\addtocounter{equation}{1}
 \begin{itlemma}}{\end{itlemma}}

 \newenvironment{proposition}{\addtocounter{equation}{1}
 \begin{itproposition}}{\end{itproposition}}

 \newenvironment{definition}{\addtocounter{equation}{1}
 \begin{itdefinition}}{\end{itdefinition}}

 \newenvironment{remark}{\addtocounter{equation}{1}
 \begin{itremark}}{\end{itremark}}

 \newenvironment{claim}{\addtocounter{equation}{1}
 \begin{itclaim}}{\end{itclaim}}

 \newenvironment{proof}{\noindent {\bf Proof.\,}
 }{\hspace*{\fill}$\qed$\medskip}

 \newenvironment{corollary}{\addtocounter{equation}{1}
 \begin{itcorollary}}{\end{itcorollary}}

 \newcommand{\be}[1]{\begin{eqnarray}\label{#1}}
 \newcommand{\ee}{\end{eqnarray}}

 \newcommand{\bl}[1]{\begin{lemma}\label{#1}}
 \newcommand{\el}{\end{lemma}}

 \newcommand{\br}[1]{\begin{remark}\label{#1}}
 \newcommand{\er}{\end{remark}}

 \newcommand{\bt}[1]{\begin{theorem}\label{#1}}
 \newcommand{\et}{\end{theorem}}

 \newcommand{\bd}[1]{\begin{definition}\label{#1}}
 \newcommand{\ed}{\end{definition}}

 \newcommand{\bcl}[1]{\begin{claim}\label{#1}}
 \newcommand{\ecl}{\end{claim}}

 \newcommand{\bp}[1]{\begin{proposition}\label{#1}}
 \newcommand{\ep}{\end{proposition}}

 \newcommand{\bc}[1]{\begin{corollary}\label{#1}}
 \newcommand{\ec}{\end{corollary}}

 \newcommand{\bpr}{\begin{proof}}
 \newcommand{\epr}{\end{proof}}

 \newcommand{\bi}{\begin{itemize}}
 \newcommand{\ei}{\end{itemize}}

 \newcommand{\ben}{\begin{enumerate}}
 \newcommand{\een}{\end{enumerate}}

\newcommand{\diam}{\, {\rm diam}_\infty\, }
 \def \ba {\begin{array}}
 \def \ea {\end{array}}

 \def \qed {{\heartsuit\hfill}}
 \def \Z {{\mathbb Z}}

 \def \N {{\mathbb N}}
 \def \P {{\mathbb P}}

 \def \G {{\Gamma}}
 \def \L {{\Lambda}}
 \def \k {{\kappa}}
 \def \a {{\alpha}}
 \def \b {{\beta}}
 \def \e {{\epsilon}}
 
 \def \r {{\rho}}

 \def \h {{\eta}}
 \def \s {{\sigma}}

 \def \t {{\tau}}

\def\boot{\text{bootstrap}}
\newcommand{\dinf}{\, {\rm d}_{\infty}\, }

\let\a=\alpha \let\b=\beta     \let\e=\varepsilon
  \let\h=\eta    \let\k=\kappa  
            
\let\r=\rho  \let\s=\sigma \let\t=\tau   
 \let\x=\xi 
   \let\G=\Gamma  \let\L=\Lambda 

\def \sandw {\widetilde \s^{=}}
\def \floo {\widetilde \s^{-}}

\def \ses {{\hbox{\footnotesize\rm SES}}}

\def \qed {{\square\hfill}}

\def \SES {{\hbox{\footnotesize\rm SES}}}

\def\tonda#1{\left( #1 \right)}

\def\sgraffa#1{\{ #1 \}}
\def\nota#1{\marginpar{\tiny #1}}

\def\Acapo{\nonumber\\}
\def\eqref#1{(\ref{#1})}

\begin{document}

\title{A $d$--dimensional nucleation and growth model}

 \author{Rapha\"el Cerf 
 \footnote{
 Universit\'e de Paris-Sud, Math\'ematique, 
B\^at. 425, 91405 Orsay Cedex,
France
\newline
 \indent\quad
E-mail: rcerf@math.u-psud.fr}\,\,\, and Francesco Manzo
\footnote{\noindent
Universit\`a di Roma ``Tor Vergata'',
dipartimento di matematica, via della ricerca scientifica 00133 Roma Italy,
E-mail: manzo.fra@gmail.com }
}

\maketitle

\begin{abstract}
	We analyze the relaxation time of a ferromagnetic
	$d$--dimensional growth model on the lattice.
The model is characterized by $d$ parameters which represent the 
activation energies of a site, depending on the number of occupied 
nearest neighbours.
This model is a natural generalisation of the 
model studied 
by Dehghanpour and Schonmann 
\cite{DS1}, where the activation energy of a site with more
than two occupied neighbours is zero.


\end{abstract}

\section{Introduction}
Growth models have been extensively studied in many cases of 
physical relevance. Our model can be obtained with
a particular choice of the parameters for
Richardson's model on the lattice \cite{Rich} and it is closely
related to the models studied by Eden \cite{Eden},
Kesten and Schonmann \cite{KS}, and specifically
Dehghanpour and Schonmann \cite{DS1}, with which it shares
the same physical motivation, i.e.,
the study of the relaxation from a 
metastable state to the stable phase of a thermodynamic 
ferromagnetic system.  
In many  physical  cases, this event is triggered by the formation,
growth and coalescence of many droplets of the stable
phase in the midst of the metastable one.
The model we study in this paper is inspired by the metastable behavior
of the kinetic Ising model in the infinite-volume regime for 
small magnetic field and vanishing
temperature.
This regime
was studied by Dehghanpour and Schonmann 
in the two dimensional case \cite{DS2}.
The main ideas were presented in 
a simplified model in \cite{DS1}. We study here the 
model 
corresponding to the $d$--dimensional case.
There are several problems to extend the approach of Dehghanpour and
Schonmann when there are more than two activation energies.
One of them is to control the
speed of growth of large supercritical droplets. 
In the model with two activation energies,
this was achieved with the technology of the ``chronological paths''
introduced by Kesten and Schonmann \cite{KS}. We did not manage to adapt
this technology to deal with the three dimensional Ising model. In this paper,
we present an alternative new strategy to control the speed of growth.
This strategy relies on coupling arguments, where we consider specific
boundary conditions called sandwich boundary conditions, as was done
to analyze the bootstrap percolation model \cite{CeCi, CM}.
We hope to apply this strategy to control the growth of the supercritical droplets in the context of the three dimensional
Ising model 
in the regime of low temperatures.

The model is an irreversible gas model on the lattice
$\Z^d$.
Sites are occupied at exponential times 
with rates that depend on the number of occupied neighbors.
More precisely, our model is characterized by a 
set of parameters $\G_n$, $n=0,\dots, d$ that represent
the activation energy of a ``critical droplet" in
dimension $n$. 
When a site has $i\leq d$ occupied neighbors,
its occupation rate is $\exp(-\b \G_{d-i})$. 
When a site has $d$ or more occupied neighbors,
its occupation rate is $1$.
A natural choice for ferromagnetic systems is to assume
$$\G_0\,\leq\,
\Gamma_1\,\leq\,\cdots\,\leq\, \G_d\,.$$
We start from the void configuration in infinite volume or in
a finite cube and look at the 
time $\t_d$ when a given site, for instance the origin, is occupied.
The scaling behavior of $\tau_d$ as $\beta$ goes to $\infty$
can be obtained with the help of the following simplified heuristics.
The rate of creation of nuclei (namely, isolated occupied sites) is 
$\exp(-\b \G_{d})$. 
Once a nucleus has appeared, it starts to grow, yet its speed of growth
increases with its size. Let $l(\tau)$ be the typical diameter of
a droplet grown from a nucleus after a time $\tau$. At time $2\tau$,
the origin is likely to have been reached by any nucleus created at distance
$l(\tau)$ before time $\tau$. The relaxation time $\tau_d$ 
should be such that the rate of creation of a nucleus within
the space time cone
$
l(\tau_d)^d\times\tau_d$
is of order one. It turns out that 
$l(\tau)$ behaves as follows when
$\beta$ goes to $\infty$:
$$
l\big(\exp(\beta K)\big)\,\sim\,
\left\{
\begin{matrix}
  1 &\quad
  \text{if $K<\Gamma_{d-1}$}
\\
\,\,\exp\big(\beta( K-\kappa_{d-1})\big)
   &\quad
  \text{if $K\geq\Gamma_{d-1}$}
\\
\end{matrix}
\right.
$$
Suppose that $\tau_d$ scales as 
$\exp(\beta \kappa_d)$
when
$\beta$ goes to $\infty$.
The value $\kappa_d$ will be
the smallest value $K$ such that
$$
l\big(\exp(\beta K)\big)^d
\exp(\beta K)
\exp(-\beta \Gamma_d)
$$
is of order $1$.
Since 
$\Gamma_d\geq
\Gamma_{d-1}$, then 
$K$ has to be larger than
$\Gamma_{d-1}$, and it satisfies therefore
$$
\exp\big(d\beta( K-\kappa_{d-1})\big)
\exp(\beta K)
\exp(-\beta \Gamma_d)
\,=\,1\,.
$$
This equation yields
$$K\,=\,
\frac{\Gamma_d+d\kappa_{d-1}}{d+1}\,.
$$
We conclude finally that
$$\kappa_d\,=\,\max\Big(
\Gamma_{d-1},
\frac{\Gamma_d+d\kappa_{d-1}}{d+1}
\Big)\,.
$$
\section{Main result}
Our configuration space is 
$\sgraffa{0,1}^{\Lambda}$, where $\Lambda$ is a subset of $\Z^d$
(possibly equal to $\Z^d$ itself).
A configuration is thus a map $\sigma:\Lambda\to \{\,0,1\,\}$,
and a site $x\in\Lambda$ is empty 
(respectively occupied) 
in the configuration $\sigma$
if $\sigma(x)=0$ (respectively $\sigma(x)=1$). 
Sites which are occupied remain occupied forever. 
To define the dynamics,
we consider a family of i.i.d. Poisson processes with rate one, associated 
with the sites in $\Z^d$.
For $x\in \Z^d$, $i \geq 1$, we denote by $\t(x,i)$
the $i$--th arrival time of the Poisson process associated with $x$.
With each arrival time, we associate a uniform random variable
$U(x,i)$ in $[0,1]$, independent of the Poisson processes 
and of the other uniform variables. 
We build a Markov process $(\sigma_{\L,t})_{t\geq 0}$
with the help of these random objects.  
At time $0$, we start from the empty configuration:
$$\forall x\in\Z^d\qquad \sigma_{\L,0}(x)=0\,.$$
We describe now the updating procedure of our process.
Let $N(x,\sigma)$ be the number of occupied neighbors of the site $x$ 
in the configuration~$\sigma$, i.e.,
$$N(x,\sigma)\,=\,
\sum_{y\in\L:|x-y|=1}\sigma(y)\,.$$
The rate at which a site becomes occupied depends only on the number of
its occupied neighbors. These rates are given by a non--decreasing
sequence
$$c(0)\,\leq\,c(1)\,\leq\,\cdots\,\leq\,c(2d)\,.$$
A site $x$ can become occupied only at a time corresponding to an
arrival of its associated Poisson process. 
Suppose that $t=\t(x,i)$ for some $i\geq 1$ and that $x$ was not occupied
before time $t$. With probability one, all the arrival times are distinct
and only the state of the site $x$ can change at time $t$.
If
$$U(x,i)\,\leq\,c\big(N(x,
\sigma_{\L,t}(x))\big)$$
then $x$ becomes occupied at time~$t$, otherwise it stays vacant.
If the set $\L$ is finite, the above rules define 
a Markov process $(\sigma_{\L,t})_{t\geq 0}$.
Whenever $\L$ is infinite, one has to be more careful, because
there is an infinite number of arrival times in any finite
time interval and 
it is not possible to order them in an increasing sequence.
However, because the rates are bounded, changes in the system
propagate at a finite speed, and 
a Markov process 
can still
be defined by taking the limit of finite volume processes
(see \cite{L} for more details).
Whenever $\L=\Z^d$, we drop it from the notation, and we
write
$(\sigma_{t})_{t\geq 0}$ for the infinite volume process in $\Z^d$. 
We will deal with exponentially small rates. However we need to have
a sufficiently loose asymptotic condition in order to perform our
inductive proof, so that we can compare the process in dimension $d$ with a
$d-1$ dimensional process satisfying the same condition.
\medskip

\noindent
{\bf Hypothesis on the rates.}
We suppose that the occupation rates $c(n)$, $0\leq n\leq 2d$,
depend on a parameter $\beta>0$ and that
the following limits exist:
\begin{align*}\label{hr1}
\forall n\in\{\,0,\dots,d\,\}\qquad
&\lim_{\beta\to\infty}\,\,
\frac{1}{\beta}\ln c_\beta(n)\,=\,
-\G_{d-n}\,,\cr
\forall n\in\{\,d,\dots,2d\,\}\qquad
&\lim_{\beta\to\infty}\,\,
\frac{1}{\beta}\ln c_\beta(n)\,=\,
0\,.
\end{align*}
Moreover, we suppose that 
$$\G_0\,\leq\,
\Gamma_1\,\leq\,\cdots\,\leq\, \G_d\,.$$
%
For $0\leq n\leq d$, 
the parameter $\G_n$ represents
the activation energy of a critical droplet in
dimension~$n$.
The conditions imposed on the sequence 
$\G_n$, $0\leq n\leq d$,
simplify substantially the analysis and
they are satisfied by the growth model associated to the 
metastability problem for the
low--temperature Ising model.
We define a sequence of critical constants $\kappa_i$ for
$0\leq i\leq d$ by setting $\kappa_0=\Gamma_0$ and
$$\forall i\in\{\,1,\dots,d\,\}\qquad
\kappa_i\,=\,\max\Big(
\Gamma_{i-1},
\frac{\Gamma_i+i\kappa_{i-1}}{i+1}
\Big)\,.
$$
Thus we have
$$\displaylines{ 
\kappa_d\,=\,
\max\Big(
{\Gamma_{d-1}},
\frac{\Gamma_d+d\Gamma_{d-2}}{d+1},\dots,
\frac{\Gamma_d+\cdots+
\Gamma_{d-i}+
(d-i)\Gamma_{d-i-2}}{d+1},
\dots,
\hfill\cr
\frac{\Gamma_d+\cdots+
\Gamma_{3}+
3\Gamma_{1}
}{d+1},
\frac{\Gamma_d+\cdots+
\Gamma_{2}
+2\Gamma_{0}
}{d+1},
\frac{\Gamma_d+\cdots+
\Gamma_{1}
+\Gamma_{0}
}{d+1}
\Big)\,.
}$$
Our main result states that, in infinite volume, the relaxation time
of the system scales as
$\exp(\beta \k_d)$.

\bt{T1iv}{(\bf Infinite volume.)}
Let $\kappa>0$ and let $\tau_\beta=\exp(\beta \k)$.
\par\noindent
$\bullet$ If $\kappa<\kappa_d$, then
$$\lim_{\beta\to\infty}\,\,
	  \P\tonda{\s_{\tau_\beta}(0)=1}\,=\,0\,.$$
\par\noindent
$\bullet$ If $\kappa>\kappa_d$, then
$$\lim_{\beta\to\infty}\,\,
	  \P\tonda{\s_{\tau_\beta}(0)=0}\,=\,0\,.$$
\et
The first step of the proof consists in reducing the problem to some
growth processes in a finite volume.
Indeed, 
if 
$\kappa<K$ and we set 
\[
\tau_\beta=\exp(\beta \k)\,,\qquad
  \L_\beta=\L(\exp{\b K})
  \,,\]
then
\begin{equation*}
\lim_{\beta\to{\infty}}\,\,
    \P \tonda{\s_{\tau_\beta}(0) \, =\, \s_{\L_\beta,\tau_\beta}(0)}\, =\, 1\,.
\end{equation*}
This follows from
a simple large-deviation estimate
based on the fact that the maximum rate in the model is~$1$, 
see lemma~$1$ of \cite{DS2}
for the complete proof. 
Let us shift next our attention to finite volumes.
We have two possible scenarios for the growth process 
in order to fill completely a cube.
If the cube is small,
the system relaxes via the formation of a single nucleus
that grows until filling the entire volume. 
If the cube is large, a more efficient mechanism 
consists in
creating many droplets that grow and eventually coalesce.
The critical side length of the cubes separating these two mechanisms
scales exponentially
with $\b$ as $\exp(\b L_d)$, where
$$L_d\,=\,
\frac{\Gamma_d-\kappa_{d}}{d}
\,.$$
There are three main factors controlling the relaxation time:
\medskip

\noindent
{\bf Nucleation.}
Within a 
box of sidelength $\exp(\beta L)$, the typical time when
the first nucleus appears is of order
$\exp(\beta (\Gamma_d-dL))$.

\noindent
{\bf Initial growth.}
The typical time to grow a nucleus into a droplet travelling
at the asymptotic speed
is
$\exp(\beta \Gamma_{d-1})$.

\noindent
{\bf Asymptotic growth.}
A droplet travelling
at the asymptotic speed 
covers a region of diameter
$\exp(\beta L)$ in a time
$\exp(\beta (L+\kappa_{d-1}))$.
\medskip

\noindent
The statement concerning the nucleation time contains no mystery.
Let us try to explain the statements on the growth of the droplets.
Once a nucleus is born, it starts to grow at speed
$\exp(-\beta \Gamma_{d-1})$. As the droplet grows, the speed of growth
increases, because the number of choices for the creation of a new
protuberance attached to the droplet is of order the surface of the droplet.
Thus the speed of growth of a droplet of size
$\exp(\beta K)$ is
$$\exp(\beta (K(d-1)-\Gamma_{d-1}))\,.$$
When $K$ reaches the value $L_{d-1}$, the speed of growth is limited
by the time needed for the protuberance to cover an entire face of the
droplet. This time corresponds to the $d-1$ relaxation time and the droplet
reaches its asymptotic speed, of order
$\exp(-\beta \kappa_{d-1})$.
The time needed  
to grow a nucleus into a droplet travelling
at the asymptotic speed
is
$$\sum_{1\leq i\leq
\exp(\beta L_{d-1})}
\exp
\beta\Big(
\Gamma_{d-1}-
\frac{d-1}{\beta}\ln i
\Big)$$
and it is still of order
$\exp(\beta \Gamma_{d-1})$.
With the help of the above facts, we can obtain easily an upper bound
on the relaxation time in a
box $\Lambda_\beta$ of sidelength $\exp(\beta L)$.
Indeed, the relaxation time is smaller than
the sum
$$\displaylines{ 
\left(
\begin{matrix}
\text{time for nucleation}\\
\text{in the box $\L_\beta$
}
\end{matrix}
\right)
\,+\,
\left(\!
\begin{matrix}
\text{
time to grow a nucleus 
}\\
\text{
into a droplet 
travelling
}\\
\text{
at the asymptotic speed
}
\end{matrix}
\!\right)
\,+\,
\left(
\!
\begin{matrix}
\text{
time to cover 
}\\
\text{
the box $\Lambda_\beta$
}
\end{matrix}
\!\right)
\cr
\,\sim\,
\exp(\beta (\Gamma_{d}-dL))\,+\,
\exp(\beta \Gamma_{d-1})
\,+\,
\exp(\beta (L+\kappa_{d-1}))
\,
}$$
which is of order
$$\exp\Big(\beta 
\max\big(\Gamma_{d}-dL,
\Gamma_{d-1},
L+\kappa_{d-1}
\big)\Big)
\,.
$$
Optimizing 
over the size of the box $\Lambda_\beta$, we
conclude that 
the relaxation time in infinite volume
satisfies
\begin{equation*}
\tau_d\,\leq\,\exp\Big(
\beta\inf_L\,\max\big(\Gamma_{d}-dL,
\Gamma_{d-1},
L+\kappa_{d-1}
\big)\Big)
\,.
\end{equation*}
Let us now try
to obtain a lower bound
on the relaxation time. 
Suppose that we examine the state of the origin 
at a time
$\exp(\beta \kappa)$. 
The origin becomes occupied when it is covered
by a droplet. This droplet can result either from the growth of a single
nucleus or from the coalescence of several droplets. 
Since the speed of propagation of the effects is finite, the state of
the origin 
at time
$\exp(\beta \kappa)$ 
is unlikely to have been influenced by any event occurring
outside the box
of sidelength $\exp(2\beta \kappa)$. 
Thus all the subsequent computations
can be restricted to this box. 
In particular, a droplet which covers
the origin 
before time
$\exp(\beta \kappa)$ 
has to be born inside this box, meaning that the oldest site 
of the droplet belongs to this box.
Let us consider
the box $\Lambda_\beta$
of sidelength $\exp(\beta L)$.
We can envisage two scenarios.
If the droplet which covers the origin is born inside
the box $\Lambda_\beta$, then nucleation has occurred inside this box.
If the droplet which covers the origin is born outside
the box $\Lambda_\beta$, then it has grown into a droplet of diameter
at least
$\frac{1}{2}\exp(\beta L)$ in order to reach the origin.
Thus the relaxation time is larger than
$$\displaylines{ 
\min\left(
\left(
\begin{matrix}
\text{time for nucleation}\\
\text{in the box $\L_\beta$
}
\end{matrix}
\right)
\,,\,
\left(
\begin{matrix}
\text{
time to grow a nucleus 
into}\\
\text{
a droplet 
of diameter
$\frac{1}{2}\exp(\beta L)$
}\\
\end{matrix}
\right)
\right)
\cr
\,\sim\,
\min\Big(
\exp(\beta (\Gamma_{d}-dL))\,,\,
\exp(\beta \Gamma_{d-1})
\,+\,
\frac{1}{2}
\exp(\beta (L+\kappa_{d-1}))\Big)
}$$
which is of order
$$\exp\Big(\beta 
\min\big(\Gamma_{d}-dL,
\max(
\Gamma_{d-1},
L+\kappa_{d-1}
)\big)\Big)
\,.
$$
By optimizing 
over the size of the box $\Lambda_\beta$, we
conclude that 
the relaxation time in infinite volume
satisfies
\begin{equation*}
\tau_d\,\geq\,\exp\Big(
\beta
\sup_L\,
\min\big(\Gamma_{d}-dL,
\max(
\Gamma_{d-1},
L+\kappa_{d-1}
)\big)
\Big)
\,.
\end{equation*}
Since the optimal value of $L$ solves 
$\G_d-dL=L+\k_{d-1}$,
the two constants appearing in the exponential
in the lower and upper bounds for the relaxation time
coincide, they are equal to
$$\kappa_d\,=\,\max\Big(
\Gamma_{d-1},
\frac{\Gamma_d+d\kappa_{d-1}}{d+1}
\Big)\,.
$$
We state next precisely the finite volume results
that we will prove.
\medskip

\noindent {\bf Terminology.}
We say that a probability $\P(\cdot)$ is exponentially small
in $\beta$ (written in short ES) if it satisfies
$$\limsup_{\beta\to\infty}\,\,
\frac{1}{\beta}\ln
	  \P(\cdot)\,<\,0\,.$$
We say that a probability $\P(\cdot)$ is super--exponentially small
in $\beta$ (written in short SES) if it satisfies
$$\lim_{\beta\to\infty}\,\,
\frac{1}{\beta}\ln
	  \P(\cdot)\,=\,-\infty\,.$$
\bt{T1fv}{(\bf Exponential volume.)}
Let $L>0$ and let $\L_\beta=\L(\exp(\beta L))$ be a cubic
box of sidelength $\exp(\beta L)$.
Let $\kappa>0$ and let $\tau_\beta=\exp(\beta \k)$.
\par\noindent
$\bullet$ If $\kappa<\max(\Gamma_d-dL,\kappa_d)$, then
\begin{equation*}
\lim_{\beta\to\infty}\,\,
	  \P\tonda{\s_{\L_\beta;\tau_\beta}(0)=1}\,=\,0\,
\end{equation*}
and this probability is exponentially small in $\beta$.
\par\noindent
$\bullet$ If $\kappa>\max(\Gamma_d-dL,\kappa_d)$, then
\begin{equation*}
\lim_{\beta\to\infty}\,\,
\P\tonda{
	  \exists\, x\in\L_\beta
	  \quad
	  \s_{\L_\beta;\tau_\beta}(x)=0}\,=\,0\,
\end{equation*}
and this probability is super--exponentially small in $\beta$.
\et
%
%
The hardest part of theorem~\ref{T1fv} is the upper bound on the relaxation
time, i.e., the first case
where $\kappa<
\max(\Gamma_d-dL,\kappa_d)$.
The first ingredient in the proof is a lower bound on the time needed to
create a large droplet.
\begin{proposition} 
\label{inigro}
Let $L>0$ and let $\L_\beta=\L(\exp(\beta L))$ be a cubic
box of sidelength $\exp(\beta L)$.
Let $\kappa<\Gamma_{d-1}$
and let $\tau_\beta=\exp(\beta \k)$.
The probability that an occupied cluster in
	  $\s_{\L_\beta;\tau_\beta}$
has diameter larger than $\b$ is super--exponentially
small in $\beta$.
\end{proposition} 
The key result for the inductive proof
is the following control on the size of the clusters
in the configuration.
We set
$$L_d\,=\,
\frac{\Gamma_d-\kappa_{d}}{d}
\,.$$
\bt{T2}
Let $L>0$ and let $\L_\beta=\L(\exp(\beta L))$ be a cubic
box of sidelength $\exp(\beta L)$.
Let $\kappa<\kappa_d$
and let $\tau_\beta=\exp(\beta \k)$.
The probability that an occupied cluster in
	  $\s_{\L_\beta;\tau_\beta}$
has diameter larger than $\exp({\b L_d})$ is super--exponentially
small in $\beta$.
\et

%
%

By using theorem~\ref{T2} inductively, we are able to show that the
asymptotic speed of the droplets inside the box
$\Lambda_\beta$ is of order
$\exp(-{\b \kappa_{d-1}})$.
The proofs of proposition~\ref{inigro} 
and of theorem~\ref{T2} involve both a bootstrap argument to
control the coalescence of the droplets. In fact, one could make a
general statement to control the maximal size of an occupied cluster at
a given time. Yet it turns out that only the initial growth and
the asymptotic speed of the droplets 
are relevant to compute the relaxation time,
the intermediate stage of growth of the droplets is not a limiting factor.

\section{Graphical construction}
Throughout the paper, we use the standard graphical construction
\cite{DS2}.
All our processes are defined on the same probability space and
they are built with the help of the arrival times of independent
Poisson processes
and the associated uniform random variables
$$\t(x,i)\,,\quad 
U(x,i)\,,\quad i\geq 1\,,\quad x\in\Z^d\,.$$
This provides a natural coupling between the
different growth processes.
The process in a set $\Lambda$ with boundary conditions $\rho$
is denoted by
$$(\sigma_{\L,t}^\rho)_{t\geq 0}\,.$$
This coupling preserves the natural order 
on the configurations. A configuration $\alpha$
is included in 
a configuration $\rho$, which we denote by
$\alpha\leq\rho$, if every site occupied in $\alpha$
is also occupied in $\rho$.
The growth process in a box $\Lambda$
starting from the configuration $\alpha$ will always
remain smaller
than
the growth process in $\Lambda$
starting from a larger configuration $\rho$.
The growth processes in a box $\Lambda$ associated to
different boundary conditions are also coupled in the same
way, and the coupling respects the order on the boundary conditions,
meaning that larger boundary conditions lead to larger growth processes.
We rely repeatedly on this coupling in order to compare our model
with simpler or lower-dimensional processes. 


\section{Bootstrap}

Following \cite{DS2}, we control the effect of the coalescence 
of the droplets with a bootstrap-percolation argument.
  We recall next the standard bootstrap procedure.
Let $A$ be a finite subset of $\Z^d$.
We start with a configuration $\h \in \sgraffa{0,1}^{A}$
and we occupy iteratively all the sites which have at least
two occupied neighbors, until exhaustion.
  Since the procedure is monotonic and the volume is finite,
  the algorithm will stop after a finite number of steps. 
  We denote by $\boot(\h)$ 
the final configuration obtained by bootstraping $\h$.
This final configuration is an union of occupied parallelepipeds,
which are pairwise at distance larger than or equal to two. 
Following \cite{AL}, we say that a set $E\subset{\mathbb Z}^d$ is
internally spanned in the configuration~$\eta$
if it is entirely covered in the 
final configuration of the dynamics restricted to $E$. More
precisely, the initial configuration is the restriction
of $\eta$ to $E$ and 
the dynamics runs on the sites of $E$
without taking into account sites outside $E$.

We will use the supremum norm, given by
$$\forall x=(x_1,\dots,x_d)\in\Z^d\qquad
|x|_\infty\,=\,
\max_{1\leq i\leq d} |x_i|\,.$$
We denote by
$\dinf$ the distance associated to the supremum norm
and we define the 
$\dinf$ diameter
$\diam C$
of a subset $C$ of ${\mathbb Z}^d$ by
$$
\diam C\,=\,\sup\,\big\{\,|x-y|_\infty:\; x,y\in C
\,\big\}\,.$$
Thus $\diam C$ is the sidelength of the minimal cube
surrounding~$C$. 
The following lemma 
is a key observation 
of Aizenman and Lebowitz
\cite{AL}. 
\bl{AL}
  If a set $C$ is internally spanned in a configuration~$\eta$
then
  for all integer $k\geq 1$ such that
$2k+1 <\diam C$ there exists 
a subset $D$ of $C$ which is internally spanned
in $\eta$ and such that
  $k\leq  \diam D\le 2k + 1$.
\el
We give the 
sketch of the proof, which can be found in \cite{AL}. 
It relies on the fact that if $\h \le \xi \le 
\boot(\h)$, then $\boot(\xi)=\boot(\h)$.
For this reason, we are free to change the updating 
order without affecting the final configuration.
The idea is then to realize the bootstrap percolation 
by occupying a single site at each  step.
If the maximal diameter of
the clusters present in the configuration is $k$ 
before one step of the algorithm, then
right after occupying one site, the new maximal 
diameter is between $k$ and $2k + 1$.
Looking at the evolution of the maximal diameter 
of the occupied clusters, we get
the thesis.

\section{Proof of proposition \ref{inigro}}
Let $L>0$ and let $\L_\beta=\L(\exp(\beta L))$ be a cubic
box of sidelength $\exp(\beta L)$.
Let $\kappa<\Gamma_{d-1}$
and let $\tau_\beta=\exp(\beta \k)$.
Let $\alpha$ be the random configuration defined as follows.
For $x\in \L_\beta$, we set
$\alpha(x)=1$ if there exists $i\geq 1$ such that
$\tau(x,i)\leq\tau_\beta$ and
$U(x,i)\leq c_\beta(1)$, otherwise we set
$\alpha(x)=0$.
The law of the configuration $\alpha$ is the Bernoulli product law
with parameter
$p_\beta$ given by
$$
p_\beta\,=
1-\exp\big(-c_\beta(1)\tau_\beta\big)
\,.
$$
Taking logarithm,
we see that
$$\lim_{\beta\to\infty}\,\,
\frac{1}{\beta}\ln
p_\beta\,=\,-\Gamma_{d-1}+\kappa\,<0.$$
Let $\boot(\alpha)$ be the configuration obtained by bootstraping $\alpha$.
The configuration 
	  $\s_{\L_\beta;\tau_\beta}$
	  is smaller than or equal to $\boot(\alpha)$.
	  Indeed, in order to grow beyond
	  $\boot(\alpha)$, the process would have to occupy
	  a site outside 
	  $\boot(\alpha)$ having $0$ or $1$ occupied neighbors, but
	  all these events until time $\tau_\beta$ were already recorded
	  in the initial configuration $\alpha$.
Proposition \ref{inigro}
	  is therefore implied by
	 the following lemma.
\begin{lemma}
  \label{booini}
  The probability that
there exists an occupied cluster
    in the configuration
$\boot(\alpha)$ 
    whose
$\dinf$ diameter
is larger than $\b$
is super--exponentially
small in $\beta$.
\end{lemma}
\begin{proof}
We say that a box is crossed if, after applying 
the bootstrap operator restricted to the box, there is an 
occupied connected set joining two opposite faces of the box.
By lemma \ref{AL}, 
if there is an occupied cluster
    in $\boot(\alpha)$ whose
$\dinf$ diameter
is larger than $\b$, then
there exists an
  internally--spanned cluster in $\boot(\alpha)$ with diameter between 
$\b$ and
$2\b+1$.
  Let $Q_\beta$ be a cube of minimal side length containing such a cluster.
The cube 
  $Q_\beta$ has to be crossed in one of the $d$ directions parallel to the axis,
say for instance the vertical one.
If there is an horizontal strip in $Q_\beta$ of height 
$2$ which is void in the configuration $\alpha$
then the box
$Q_\beta$ cannot be crossed vertically.
Thus
$$\displaylines{
\P \big(\text{$Q_\beta$ is crossed vertically}
\big)
\hfill\cr
\,\leq\,
\P
\left(
\begin{matrix}
\text{each horizontal strip in $Q_\beta$ of height 
$2$ is}
\\
\text{
non void in the initial configuration $\alpha$ 
}
\end{matrix}
\right)
\cr
\,\leq\,
\P
\left(
\begin{matrix}
\text{one fixed horizontal strip in $Q_\beta$ of height 
$2$
}\\
\text{
is non void in the initial configuration $\alpha$ 
}
\end{matrix}
\right)^{\beta/2-1}
\cr
\,\leq\,
\big(1-(1-p_\beta)^{2(2\beta+1)^{d-1}}\big)^{\beta/2-1}\,.
}$$
To complete the estimate, we count the number of possible choices 
for the box $Q_\beta$:
$$\displaylines{
\P
\left(
\begin{matrix}
\text{there is an occupied cluster
    in
$\boot(\alpha)$ 
}\\
\text{whose
$\dinf$ diameter
is larger than $\b$}
\end{matrix}
\right)\hfill\cr
\,\leq\,
|\L_\beta| \times 3\beta
\times d\,
 \P \big(\text{$Q_\beta$ is crossed vertically}
\big)
\cr
\,\leq\,
3d\beta\exp(\beta dL)
\big(1-(1-p_\beta)^{2(2\beta+1)^{d-1}}\big)^{\beta/2-1}
}$$
and this last bound is SES.
\end{proof}

\section{Proof of theorem \ref{T2}}

  In this section we prove
  theorem \ref{T2}
with the help of an induction over the dimension~$d$.
  The main point here is the bound on the asymptotic speed of growth of 
  a droplet. 
  Our approach gives a bound on the
  probability of a ``too fast" growth. 
  Since this bound is super-exponential, while both the volume and
  the time we are considering are exponential,
  we end up with a deterministic computation rather than a 
  large-deviation estimate as in \cite{DS2}.
  This fact allows to avoid all combinatorial 
  problems like counting the number of ``chronological paths" and
  it is the main technical difference 
  with the method used in \cite{DS2}.
  Heuristically, the process evolves as if
the droplets were growing
  one shell after the other,
  filling the sites on one face before passing to the next.
  Since all the sites on a face are neighbors of an
  occupied site in the droplet, this growth mechanism is 
  analogous to a nucleation and growth mechanism in dimension $d-1$.
  We use the 
  $d-1$ dimensional bound on the size of the clusters
  to show that, up to SES events,
 a too--fast growth has to take place 
  into a parallelepiped with ``small" base. 
This is a SES bound, and the result holds in
  any exponential volume.
  Throughout the section, we
let 
$$\L_\beta=\L(\exp(\beta L))$$ 
be a cubic
box of sidelength $\exp(\beta L)$, where $L>0$.
Let $\kappa<\kappa_{d}$
and let $\tau_\beta=\exp(\beta \k)$.
Coalescence is a nontrivial effect only if $L \ge L_d$,
since otherwise the number of droplets formed 
in $\Lambda_\beta$ before time $\tau_\beta$ is finite.
Theorem \ref{T2} needs to be proved only for
$L \ge L_d$.

\subsection{Dilation, bootstrap and erosion}
The procedure we are going to define is a modified 
version of standard bootstrap percolation and
is specifically suited to our setting.
The same results can be obtained by rescaling the lattice as in \cite{DS2}
and using the standard bootstrap percolation arguments
developed in \cite{AL,CM}.
\def\dilate{\text{dilate}}
\def\erode{\text{erode}}
We denote by $\dinf$ the distance associated to the supremum norm, 
given by
$$\forall x,y\in\Z^d\qquad
\dinf(x,y)\,=\,|x-y|_\infty\,=\,
\max_{1\leq i\leq d} |x_i-y_i|\,.$$
Let $\Lambda$ be a subset of $\Z^d$,
let
$\eta$ 
be a configuration in
$\sgraffa{0,1}^{\Lambda}$
and let $l\geq 0$. 
We define the dilated configuration
$\dilate(\eta,l)$ by occupying all the sites of $\L$ which are
at a $\dinf$ distance strictly less than~$l$ from a site occupied in~$\eta$:
$$
\forall x\in
    \L\quad
\dilate(\eta,l)(x)
\,=\, 
\begin{cases}
1\quad\text{ if }\quad
\exists\,y\in\L\quad \dinf(x,y)<l\,,\quad\eta(y)=1\\
0\quad\text{ otherwise }
\end{cases}
$$
We define the eroded configuration
$\erode(\eta,l)$ by emptying all the sites of $\L$ which are
at a $\dinf$ distance strictly less than~$l$ from an empty site in~$\eta$:
$$
\forall x\in
    \L\quad
\erode(\eta,l)(x)
\,=\, 
\begin{cases}
0\quad\text{ if }\quad
\exists\,y\in\L\quad \dinf(x,y)<l\,,\quad\eta(y)=0\\
1\quad\text{ otherwise }
\end{cases}
$$
Dilation and erosion are classical operations in 
mathematical morphology.

Let $\eta$ be the random configuration defined as follows.
For $x\in \L_\beta$, we set
$\eta(x)=1$ if there exists $i\geq 1$
such that
$\tau(x,i)\leq\tau_\beta$ and
$U(x,i)\leq c_\beta(0)$, otherwise
we set
$\eta(x)=0$.
The law of the configuration $\eta$ is the Bernoulli product law
with parameter
$p_\beta$ given by
$$
p_\beta\,=
1-\exp\big(-c_\beta(0)\tau_\beta\big)
\,.
$$
Taking logarithm, we see that
$$\lim_{\beta\to\infty}\,\,
\frac{1}{\beta}\ln
p_\beta\,=\,-\Gamma_{d}+\kappa\,.$$
Let $\boot(\eta)$ be the configuration obtained by bootstraping $\eta$.
\begin{proposition} 
\label{tb1}
Let $\rho$ be the configuration
obtained by dilating $\eta$ with a distance 
${\beta}^{-1}\exp(\beta L_d)$
and then bootstraping it:
$$\rho\,=\,\boot(\dilate(\eta,{\beta}^{-1}\exp(\beta L_d)))\,.$$
The probability that there is an occupied cluster
    in $\rho$ whose
$\dinf$ diameter
is larger than $\exp{(\b L_{d})}$
is super--exponentially
small in $\beta$.
\end{proposition}
\bpr
  By lemma \ref{AL}, 
if there is an occupied cluster
    in $\rho$ whose
$\dinf$ diameter
is larger than $\exp{(\b L_{d})}$, then
there exists an
  internally--spanned cluster in $\rho$ with diameter between 
$\exp{(\b L_{d})}$ and
$2\exp{(\b L_{d})}+1$.
  Let $Q_\beta$ be a cube of minimal side length containing such a cluster.
The cube 
  $Q_\beta$ has to be crossed in one of the $d$ directions parallel to the axis,
say for instance the vertical one.
Let 
$Q'_\beta$ be the parallelepiped having the same center and the same
height as $Q_\beta$ and whose
sidelengths in the other directions are three times the 
sidelength of $Q_\beta$.
If there is an horizontal strip in $Q'_\beta$ of height
$3{\beta}^{-1}
\exp{(\b L_{d})}$ which is void in the initial configuration $\eta$, 
then there is an horizontal strip in $Q_\beta$ of height 
$2$ which is void in the intermediate configuration 
$$\dilate(\eta,{\beta}^{-1}\exp(\beta L_d))$$
and the box
$Q_\beta$ cannot be crossed vertically after the bootstraping.
Thus
$$\displaylines{
\P \big(\text{$Q_\beta$ is crossed vertically in $\rho$}
\big)
\,\leq\,
\hfill\cr
\P 
\left(
\begin{matrix}
\text{each horizontal strip in $Q'_\beta$ of height 
$3\beta^{-1}\exp{\left( \b L_{d} \right)}$ 
}\\
\text{
is  non void in the initial configuration $\eta$ 
}
\end{matrix}
\right)
\cr
\,\leq\,
\P 
\left(
\begin{matrix}
\text{one fixed horizontal strip in $Q'_\beta$ of height 
$3\beta^{-1}\exp{\left( \b L_{d} \right)}$ 
}\\
\text{
is non void in the initial configuration $\eta$ 
}
\end{matrix}
\right)^{\beta/3}
\cr
\,\leq\,
\left(
1-
\exp\Big(
9^{d-1}\exp((d-1)\b L_d)\times 3\beta^{-1}\exp{\left( \b L_{d} \right)} 
\times \ln(1-p_\beta)
\Big)
\right)^{\beta/3}
\cr
\,\leq\,
\left(
-9^d\beta^{-1}{
\exp(d\b L_d)
\times \ln(1-p_\beta)
}
\right)^{\beta/3}
\,
}$$
and this last bound is SES because $dL_d+\Gamma_d=\kappa_d$.
To complete the estimate, we count the number of possible choices 
for the box $Q_\beta$:
$$\displaylines{
\P
\left(
\begin{matrix}
\text{there is an occupied cluster
    in $\rho$ whose
}\\
\text{$\dinf$ diameter
is larger than $\exp{\b L_{d}}$}
\end{matrix}
\right)\hfill\cr
\,\leq\,
|\L_\beta| \times 3\exp(\beta L_d)
\times d\,
\P \big(\text{$Q_\beta$ is crossed vertically in $\rho$}
\big)
}$$
and the last term is SES.
\epr

Let $\xi$ be the erosion
of
$\rho$ with a distance
${(2\beta)}^{-1}\exp(\beta L_d)$, i.e.,
\be{defxi}\xi\,=\,\erode(\rho,{(2\beta)}^{-1}\exp(\beta L_d))\,.\ee
Since $\rho$ was obtained after applying the bootstrap procedure,
it is a union 
of occupied parallelepipeds,
which are pairwise at distance larger than two.
After applying the erosion operator, we obtain again
an union 
of occupied parallelepipeds,
which are pairwise at distance larger than or equal to 
${(2\beta)}^{-1}\exp(\beta L_d)$.
Moreover we dilated
$\eta$ with a distance ${\beta}^{-1}\exp(\beta L_d)$ before the bootstrap,
thus the configuration $\eta$ is still included in $\xi$, so that
all the sites where nucleation has occurred before time $\tau_\beta$
are occupied in the configuration $\eta$. By attractivity of the process,
we have
$$\s_{\L_\beta,\tau_\beta}\,\leq\,
\s_{\L_\beta,\tau_\beta}^\eta\,\leq\,
\s_{\L_\beta,\tau_\beta}^\xi$$
and because of the definition of $\eta$, no nucleation occurs in the
growth process starting from $\eta$ until the time $\tau_\beta$.
We are thus able to compare 
$\s_{\L_\beta,\tau_\beta}$ with a process where nucleation events
are cancelled, that we define in the next section. 
The crucial problem is then to control the speed of growth of the droplets
and to show that, up to a SES event, the non--nucleating
process starting from $\xi$ is still included in $\rho$ at time
$\tau_\beta$.

\subsection{Control of the speed of growth}
In this section, we study the growth process where the nucleation is
cancelled and we prove
our key estimate to control the speed of growth of the droplets.
The initial speed of growth of a nucleus is
$\exp(-\beta \Gamma_{d-1})$. For a droplet of size
$\exp(\beta K)$, the speed is
$$\exp\big(\beta ((d-1)K-\Gamma_{d-1})\big)$$ 
for 
$K< L_{d-1}$ and
$\exp(-\beta \kappa_{d-1})$
for $K\geq L_{d-1}$.
It turns out that the time needed to create a droplet travelling
at the asymptotic speed is
$\exp(\beta \Gamma_{d-1})$, which is of the same order as the time
needed to grow the initial nucleus into a droplet of diameter~$\beta$.
Hence we need only to control the speed of droplets having
a diameter larger than
$\exp(\beta L_{d-1})$, which travel at the asymptotic speed.

\noindent
{\bf Non--nucleating processes.}
We define a \emph{non-nucleating} process 
$$(\st_{\L_\beta,t})_{t\geq 0}$$
associated to the rates
$$\widetilde{c}(0)\,=\,0\,,\qquad
\widetilde{c}(n)\,=\, c(n)\,,\quad 1\leq n\leq 2d\,.$$
In this process, a site cannot become occupied unless one of its neighbors 
is occupied.
The activation energies for this process are given by
$$\widetilde{\Ga}(d)\,=\,\infty\,,\qquad
\widetilde{\Ga}(n)\,=\, \Ga(n)\,,\quad 0\leq n<d\,.$$
In the sequel, 
the various processes where nucleation is suppressed are denoted with 
a tilde above the symbol of the process.

\noindent
{\bf Floor boundary conditions.}
Let $R$ be a cylinder 
with basis a $d-1$ dimensional cubic box $\L^{d-1}$ and height $H$, 
i.e., of the form
$$R\,=\,\L^{d-1}\times\{\,0,\dots,H\,\}\,.$$
    We call {\emph{floor}} of $R$ 
its bottom face
$\L^{d-1}\times\{\,0\,\}$
and {\emph{ceiling}} of $R$ its top face
$\L^{d-1}\times\{\,H\,\}$.
We call floor boundary conditions on~$R$ the boundary condition 
defined by the following configuration
$\rho$:
\begin{equation*}
\forall x\in\Z^d\qquad
\rho(x)\,=\,
\begin{cases}
1\text{ if }x\in
\L^{d-1}\times\{\,-1\,\}\\
0\text{ otherwise}\\
\end{cases}
\end{equation*}
The process $(\st^{\rho}_{R;t})_{t\geq 0}$ in $R$ with the 
floor boundary conditions is denoted by
    $$(\floo_{R;t})_{t\geq 0}\,.$$
    We say that a configuration {\emph{crosses}} $R$
    if it contains a cluster included in $R$ which
    connects the floor and the ceiling.
  \bp{lk}
    Let $d\ge 2$ and let $K>0$. 
Let $R_\beta$ be the cylinder
    $$R_\beta\,=\,\L^{d-1}(\exp{(\b K)}) \times \{\,0,\dots,\beta\,\}\,.$$
Let
$\kappa<\kappa_{d-1}$ and $\tau_\beta=\exp(\b\kappa)$.
 Suppose that theorem~\ref{T2}
has been proved in dimension $d-1$.
Then the probability that
      $\floo_{R_\beta;\tau_\beta} 
        \text{ crosses } R_\beta$
	is SES.
  \ep

  \bpr 
    We start with the case $K > L_{d-1}$ and we set  
    $$\L_\beta^{d-1}\,=\,\L^{d-1}(
\exp(\beta L_{d-1})
    )\,.$$
    We use theorem \ref{T2} to show that,
    most likely, the cluster that crosses $R_\beta$ 
    is contained in a smaller parallelepiped 
of basis
    $\L_\beta^{d-1}$, i.e., a parallelepiped which is a translate of
    $$T_\beta\,=\,\L^{d-1}_\b \times \{\,0,\dots,\beta\,\}\,.$$
    To this end, let us consider the process 
    obtained from $(\floo_{R;t})_{t\geq 0}$ by occupying all the sites
 in each non empty column, and its projection 
$(\widehat\sigma_t)_{t\geq 0}$ on the floor 
    $R_\beta$ defined
    for $t\geq0$ by
$$
\forall\widehat x\in
    \L_\beta^{d-1}\quad
\widehat\sigma_t(\widehat x)\,=\, 
\begin{cases}
0\quad\text{ if }
    \quad\floo_{R_\beta;t}
(\widehat x,i)\,=\, 0\text{ for all }
    i\in\{\,0,\dots,\beta\,\}\\
1\quad\text{ if }
    \quad\floo_{R_\beta;t}
(\widehat x,i)\,=\, 1\text{ for some }
    i\in\{\,0,\dots,\beta\,\}\\
\end{cases}
$$
The process
$(\widehat\sigma_t)_{t\geq 0}$ is a $(d-1)$-dimensional process with rates
satisfying  
$$
c_\beta(n+1)\,\leq\,
\widehat{c}_\beta(n)\,\leq\,2c_\beta(n+1)+(\beta-2)c_\beta(n)\,,\quad
0\leq n\leq d-1\,.
$$
In terms of activation energies,
$$
\widehat{\Ga}(n)\,=\, \Ga(n)\,,\quad 0\leq n\leq d-1\,.$$
The idea is to use the $(d-1)$-dimensional bounds on the size of 
$(\widehat\sigma_t)_{t\geq 0}$ and 
attractivity to bound the size of the clusters 
of $(\floo_t)_{t\geq 0}$.
Let $\text{Large}$ be the event
$$
\text{Large}\,=\,
\left\{
\begin{matrix}
\text{there is an occupied cluster
    in $\floo_{R_\beta;\tau_\beta}$ 
}\\
    \phantom{\floo_{R_\beta}}\kern-7pt
\text{whose projection on the floor
     of $R_\beta$}\kern-7pt
    \phantom{\floo_{R_\beta}}\\
\text{has a diameter larger than $\exp{(\b L_{d-1})}$}
\end{matrix}
\right\}\,.$$
    By theorem \ref{T2} in dimension $d-1$, 
since $\kappa<\kappa_{d-1}$, 
the probability that an occupied cluster in
$\widehat\sigma_{\tau_\beta}$
has diameter larger than $\exp({\b L_{d-1}})$ is SES.
Since
    the volume of $R_\beta$ is exponential,
 the probability of the event
$\text{Large}$ is SES.
We write then
$$\displaylines{
      \P ( \floo_{R_\beta;\tau_\beta} \text{ crosses } R_\beta) 
      \,\le\, \P(\text{Large}) \,+\,
      \P \big( \{\,\floo_{R_\beta;\t_\beta} \text{ crosses } R_\beta\}\setminus
\{\, \text{Large}\,\}\big) 
	\cr
\leq\,SES\,+\,
\P\left(
\begin{matrix}
    \text{there is a translate $y+ T_\beta$ of $T_\beta$ included}\\
\text{in $R_\beta$ such that
      $\floo_{y+T_\beta;\tau_\beta} \text{ crosses } y+T_\beta$}\\ 
\end{matrix}
\right)\cr
\leq\,SES\,+\,
|R_\beta|\,\,
      \P ( \floo_{T_\beta;\tau_\beta} \text{ crosses } T_\beta)\,.
}$$
In the last step, we used the fact that the model is translation invariant.
We conclude by showing that
      $$\P ( \floo_{T_\beta;\tau_\beta} \text{ crosses } T_\beta)\,$$
is also $\ses$,
reducing ourselves to the case where
$K \le L_{d-1}$.
We shall couple the process $(\floo_{T_\beta;t})_{t\geq 0}$
with floor boundary conditions in $T_\beta$
with another simpler process.

\noindent
{\bf Sandwich boundary conditions.}
     We call {\emph{slice}} 
    a parallelepiped with height $2$
    and basis $\L_\beta^{d-1}$, which is a translate of
    $$\Sigma\,=\,\L_\beta^{d-1} \times \{\,0,1\,\}\,.$$
We call sandwich boundary conditions on~$\Sigma$ the boundary condition defined by the following configuration
$\rho$:
\begin{equation*}
\forall x\in\Z^d\qquad
\rho(x)\,=\,
\begin{cases}
1\text{ if }x\in
\L^{d-1}_\beta\times\{\,-1,2\,\}\\
0\text{ otherwise}\\
\end{cases}
\end{equation*}
We denote by $(\sandw_{\Sigma;t})_{t\geq 0}$ the process in $\Sigma$
evolving with the sandwich boundary conditions $\rho$.

\noindent
{\bf Multilayer process.}
    Let us partition the cylinder $T_\beta$ into translated slices as
    $$T_\beta\,=\,\bigcup_{i=0}^{\b /2} 
\Sigma_i\,,$$
where
$$\Sigma_i\,=\,
    \L_\beta^{d-1} \times \{\,2i,2i+1\,\}\,=\,
    \Sigma + (0,\ldots,0,2i)
\,.$$
    We define the multilayer process $\smash{(\widetilde\sigma^\equiv_{T_\beta;t})_{t\geq 0}}$ 
    using the same graphical construction as 
$(\st_{T_\beta;t})_{t\geq 0}$
    but we use sandwich boundary conditions
    in each slice.           
More precisely, we set 
$$\forall i\in\{\,1,\dots,\beta/2\,\}\quad
\forall x\in\Sigma_i\quad\forall t\geq 0\qquad
    \widetilde\s^\equiv_{T_\beta;t}(x)\,=\,
\sandw_{\Sigma_i;t}(x)\,.$$ 
A key point is that, once we put sandwich boundary conditions
around each slice, the processes in the slices become independent
of each other.
Thanks to the coupling, the process
    $\smash{(\widetilde\s^\equiv_{T_\beta;t})_{t\geq 0}}$
is always above the process $\smash{(\floo_{T_\beta;t})_{t\geq 0}}$.
    Therefore, if 
$\smash{\floo_{T_\beta;\tau_\beta}}$
crosses $T_\beta$,
so does
$\smash{\widetilde\s^\equiv_{T_\beta;\tau_\beta}}$ and
    at least a nucleus must appear 
    in each slice. Thus
$$\displaylines{
      \P \big( \floo_{T_\beta;\tau_\beta} \text{ crosses } T_\beta\big)
\,\leq\,
        \P \big( \widetilde\s^\equiv_{T_\beta;\tau_\beta} \text{ crosses } T_\beta
\big) \cr
      \le\, \P \big(
\sandw_{\Sigma_i;\t_\beta} \text{ is not void } 
\text{ for }1\leq i\leq\beta/2
\big)\,\cr
      \le \,\P \big(
\sandw_{\Sigma;\t_\beta} \text{ is not void } \big)^{\beta/2-1}
\,.
}$$
Yet
$$\displaylines{
      \P \big(
\sandw_{\Sigma;\t_\beta} \text{ is void } \big)
\,=\,
      \P 
\left(
\begin{matrix}
    \text{for any $x\in\Sigma$, there is no}\\
\text{nucleation at $x$ before $\tau_\beta$}\\
\end{matrix}
\right)\cr
\,=\,\P 
\left(
\begin{matrix}
    \text{there is no nucleation}\\
\text{at the origin before $\tau_\beta$}\\
\end{matrix}
\right)^{|\Sigma|}\cr
\,=\,\Big(\exp\Big(-
c_\beta(1)\,
\tau_\beta 
\Big)\Big)^{|\Sigma|}\cr
\,=\,\exp\Big(-
2\,|\L^{d-1}_\beta|\,
c_\beta(1)\,
\tau_\beta
\Big)\cr
\,=\,\exp-\Big(2\exp\big(\beta (d-1)K
+\ln c_\beta(1)
+\beta\kappa
\big)\Big)\,.
}$$
Since 
$K \le L_{d-1}$ and
$\kappa<\kappa_{d-1}$, we have
$$\lim_{\beta\to\infty}\,\,
\frac{1}{\beta}
\big(\beta (d-1)K
+\ln c_\beta(1)
+\beta\kappa
\big)\,=\,(d-1)K-\Ga_{d-1}+\kappa\,<\,0$$
and there exists a positive constant $\delta$
such that, for $\beta$ large enough,
$$\displaylines{
      \P \big(
\sandw_{\Sigma;\t_\beta} \text{ is void } \big)
\,\geq\,
\exp-\Big(2\exp(-\beta\delta)\Big)\,.
}$$
Reporting in the previous inequality, we get
$$\displaylines{
      \P \big( \sandw_{T_\beta;\tau_\beta} \text{ crosses } T_\beta\big)
\,\leq\,
\Big(1-
\exp-\Big(2\exp(-\beta\delta)\Big)
\Big)^{\beta/2-1}
\,.\cr
}$$
Hence the above probability is also SES.
  \epr
  \bc{lj}
    Let $d\ge 2$ and let $K,L>0$. 
Let $R_\beta$ be the cylinder
    $$R_\beta\,=\,\L^{d-1}(\exp{(\b K)}) \times \{\,0,\dots,
    \exp({\b L}) 
    \,\}\,.$$
Let
$\kappa>0$ be such that $\kappa<L+\kappa_{d-1}$ and 
$\tau_\beta=\exp(\b\kappa)$.
Suppose that theorem~\ref{T2}
has been proved in dimension $d-1$.
Then the probability that
      $\floo_{R_\beta;\tau_\beta} 
        \text{ crosses } R_\beta$
	is SES.
  \ec
  \bpr
  For $i\in\N$, let $\tau_i$
  be the first time when
  a site of the layer
  $$\L^{d-1}(\exp{(\b K)}) \times \{\,i\beta\,\}$$
  becomes occupied in the process
  $(\sandw_{R_\beta;t})_{t\geq 0}$. 
  Let us set
  $$l\,=\,\left\lfloor\frac{\exp({\b L})}{\beta}\right\rfloor\,.$$
  With these definitions, we see that
      if $\sandw_{R_\beta;\tau_\beta} 
        \text{ crosses } R_\beta$,
	then $\tau_l\leq\tau_\beta$.
	Yet
	$$\tau_l\,=\,\sum_{0\leq i<l}\tau_{i+1}-\tau_i$$
	and moreover, by using the Markov property and the attractivity
	of the process, we see that, for any $i\geq 0$, the time
	$\tau_{i+1}-\tau_i$ stochastically dominates the time $\tau_1$.
	Therefore
	$$\displaylines{ 
	\P( \tau_l\leq\tau_\beta)\,\leq\,
	\P\Big( \exists i<l\quad
	\tau_{i+1}-\tau_i\leq
	l^{-1}\exp({\b \kappa})\Big)
	\cr
	\,\leq\,
	l\,\P\Big( 
	\tau_1\leq
	l^{-1}\exp({\b \kappa})\Big)
	\,.}$$
	By hypothesis, we have $\kappa-L<\kappa_{d-1}$.
	Proposition~\ref{lk} implies that this last bound is SES.
  \epr
  \subsection{Conclusion of the proof
of theorem \ref{T2}}

  We proceed now by induction over the dimension $d$.
  The case of dimension $0$ is straightforward. In this case
  the lattice $\Z^0$ is reduced to the singleton $\{0\}$ and
  $\kappa_0=\Gamma_0$, $L_0=0$. In particular, it is impossible
  to see an occupied cluster of diameter strictly larger than $0$.
%
  Let $d\geq 1$.
  Suppose that the result has been proved in dimension $d-1$.
Let $L>0$ and let $\L_\beta=\L(\exp(\beta L))$ be a $d$--dimensional cubic
box of sidelength $\exp(\beta L)$.
Let $\kappa<\kappa_d$
and let $\tau_\beta=\exp(\beta \k)$.

  We apply corollary~\ref{lj} to show that, up to a SES event,
$\s_{\L_\beta,\tau_\beta}$ is included in the configuration $\rho$.
Indeed, suppose that it is not the case.
Then the configuration
$\smash{\st_{\L_\beta,\tau_\beta}^\xi}$ is also
not included in $\rho$.
Yet the configuration $\xi$ is
an union 
of occupied parallelepipeds,
which are pairwise at distance larger than or equal to 
${(2\beta)}^{-1}\exp(\beta L_d)$ (see \eqref{defxi}), 
and the configuration $\rho$
is obtained from $\xi$ by dilating these parallelepipeds with
a distance
${(2\beta)}^{-1}\exp(\beta L_d)$.
We consider the first time and place when the process
$(\smash{\st_{\L_\beta,t}^\xi})_{t\geq 0}$ 
occupies a site not occupied in $\rho$.
This happens close to the boundary of a face $F$ of
one of the parallelepipeds $Q$
occupied in $\rho$. Let $R_\beta$ be the cylinder included in $Q$
having for basis
this face $F$ and for height
${(2\beta)}^{-1}\exp(\beta L_d)$.
By corollary~\ref{lj},
the probability that
      $\floo_{R_\beta;\tau_\beta} 
        \text{ crosses } R_\beta$
	is SES.
	Since the number of choices of times and places above is exponential
	in $\beta$, we conclude that, up to a SES event,
the configuration
$\smash{\sigma_{\L_\beta,\tau_\beta}}$ is 
included in $\rho$.
This estimate, together with
proposition~\ref{tb1}, implies
theorem~\ref{T2}.

\section{Proof of the upper bound of theorem \ref{T1fv}}
Let $L>0$ and let $\L_\beta=\L(\exp(\beta L))$ be a cubic
box of sidelength $\exp(\beta L)$.
Let $\kappa<\max(\Gamma_d-dL,\kappa_d)$
and let $\tau_\beta=\exp(\beta \k)$.
We distinguish three different cases.
\medskip

\noindent
$\bullet$ First case: $\kappa<\Gamma_d-dL$.
If the origin is occupied at time $\tau_\beta$ for
the growth process in $\L_\beta$, then a nucleation
must have taken place in the box $\L_\beta$ before
the time $\tau_\beta$, thus
$$\displaylines{
\P\tonda{\s_{\L_\beta;\tau_\beta}(0)=1}\,\leq\,
|\Lambda_\beta|\,\left(
1-\exp\big(-c_\beta(0)\tau_\beta\big)
\right)
\,.
}$$
Taking logarithm, we see that
$$\limsup_{\beta\to\infty}\,\,
\frac{1}{\beta}\ln
\P\tonda{\s_{\L_\beta;\tau_\beta}(0)=1}\,\leq\,
dL-\Gamma_{d}+\kappa\,.$$
Yet
$\kappa<\Gamma_d-dL$ and the probability that the origin is occupied
at time $\tau_\beta$ for
the growth process in $\L_\beta$ is therefore ES in $\beta$.
\medskip

\noindent
$\bullet$ Second case: $\kappa<\Gamma_{d-1}$.
Let
$\L'_\beta
=\L(3\beta)$
be a cubic
box of sidelength $3\beta$.
Suppose that the origin is occupied at time $\tau_\beta$ for
the growth process in $\L_\beta$. The droplet which has reached the origin
is either born inside the box
$\L'_\beta$ or outside of it. In the first scenario, a nucleation
event 
must have taken place in the box $\L'_\beta$ before
the time $\tau_\beta$. In the second scenario
there is an occupied cluster in
	  $\s_{\L_\beta;\tau_\beta}$
with diameter larger than $\b$. 
We have thus
$$\displaylines{ 
\P\tonda{\s_{\L_\beta;\tau_\beta}(0)=1}\,\leq\,
\P
\left(
\begin{matrix}
\text{a nucleation
event takes}\\
\text{place in $\L'_\beta$ before
$\tau_\beta$}
\end{matrix}
\right)
\hfill\cr
\,+\,
\P
\left(
\begin{matrix}
\text{there is an occupied cluster in 
	  $\s_{\L_\beta;\tau_\beta}$
    }\\
\text{whose $\dinf$ diameter
is larger than $\b$}
\end{matrix}
\right)\,.
}$$
Proceeding as in the first case, we bound the probability of a nucleation
by
$$|\Lambda'_\beta|\,
\Big(
1-\exp\big(-c_\beta(0)\tau_\beta\big)
\Big)
\,$$
which is 
ES in $\beta$ since 
$\kappa<\Gamma_{d-1}\leq \Gamma_d$.
By proposition~\ref{inigro},
the probability that an occupied cluster in
	  $\s_{\L_\beta;\tau_\beta}$
has diameter larger than $\b$ is super--exponentially
small in $\beta$. 
\medskip

\noindent
$\bullet$ Third case: $\kappa<\kappa_{d}$.
Let
$\L'_\beta
=\L\big(3\exp(\beta L_d)\big)$
be a cubic
box of sidelength 
$3\exp(\beta L_d)$.
Suppose that the origin is occupied at time $\tau_\beta$ for
the growth process in $\L_\beta$. The droplet which has reached the origin
is either born inside the box
$\L'_\beta$ or outside of it. In the first scenario, a nucleation
event 
must have taken place in the box $\L'_\beta$ before
the time $\tau_\beta$. In the second scenario
there is an occupied cluster in
	  $\s_{\L_\beta;\tau_\beta}$
with diameter larger than 
$\exp(\beta L_d)$.
We have thus
$$\displaylines{ 
\P\tonda{\s_{\L_\beta;\tau_\beta}(0)=1}\,\leq\,
\P
\left(
\begin{matrix}
\text{a nucleation
event takes}\\
\text{place in $\L'_\beta$ before
$\tau_\beta$}
\end{matrix}
\right)
\hfill\cr
\,+\,
\P
\left(
\begin{matrix}
\text{there is an occupied cluster in 
	  $\s_{\L_\beta;\tau_\beta}$
    }\\
\text{whose $\dinf$ diameter
is larger than 
$\exp(\beta L_d)$}
\end{matrix}
\right)\,.\cr
}$$
Proceeding as in the first case, we bound the probability of a nucleation
by
$$|\Lambda'_\beta|\,
\Big(
1-\exp\big(-c_\beta(0)\tau_\beta\big)
\Big)
\,.$$
Taking logarithm, we see that
$$\limsup_{\beta\to\infty}\,\,
\frac{1}{\beta}\ln
\P
\left(
\begin{matrix}
\text{a nucleation
event takes}\\
\text{place in $\L'_\beta$ before
$\tau_\beta$}
\end{matrix}
\right)
\,\leq\,
dL_d-\Gamma_{d}+\kappa
\,<0\,.$$
By proposition~\ref{inigro},
the probability that an occupied cluster in
	  $\s_{\L_\beta;\tau_\beta}$
has diameter larger than 
$\exp(\beta L_d)$
is super--exponentially
small in $\beta$. 
\medskip

\noindent
In the three cases, the probability
$$\P\tonda{\s_{\L_\beta;\tau_\beta}(0)=1}$$
is
ES in $\beta$.

\section{Proof of the lower bound of theorem \ref{T1fv}}

We prove here
part 2 of theorem \ref{T1fv}
by induction over the dimension~$d$.
Let us consider first the case $d=0$.
We have then $\kappa_0=\Gamma_0$.
The box 
$\L_\beta$ is reduced to the singleton $\{\,0\,\}$.
Let $\kappa>\kappa_0$ and let $\tau_\beta=\exp(\beta \k)$.
We have
\begin{equation*}
\P\tonda{\s_{\L_\beta;\tau_\beta}(0)=0}\,=\,
\exp-(c_\beta(0)\tau_\beta)
\,=\,
\exp-(c_\beta(0)\exp(\beta\kappa))
\,.
\end{equation*}
Since by hypothesis,
\begin{equation*}
\lim_{\beta\to\infty}\,\,
\frac{1}{\beta}\ln c_\beta(0)\,=\,
-\G_{0}\,
\end{equation*}
we conclude that the above probability is SES.
We suppose now that $d\geq 1$ and that the result has been proved
in dimension $d-1$.
Let $L>0$ and let $\L_\beta=\L(\exp(\beta L))$ be a cubic
box of sidelength $\exp(\beta L)$.
Let $\kappa>0$ and let $\tau_\beta=\exp(\beta \k)$.
Let $\varepsilon>0$.
We define the nucleation time 
$\tau_{\text{nucleation}}$
in $\Lambda_\beta$ as
\begin{equation*}
\tau_{\text{nucleation}}
  \,=\,\inf\,\big\{\,t\geq 0:\exists\, x\in\Lambda_\beta
  \quad
\s_{\L_\beta;t}(x)=1\,\big\}\,.
\end{equation*}
We have
$$
\forall t>0\qquad
\P(
  \tau^N\,>\,t)
  =
\exp\big(- |\L_\beta|\, c_\beta(0)\,t\big)\,.
$$
Therefore,
up to a SES event, the first nucleus in the box
$\L_\beta$
appeared before time
$$\exp\Big(\beta\big(
\Gamma_{d}-dL+\varepsilon\big)
\Big)\,.
$$
For $i\geq 1$, we define
the first time $\tau^i$ when there is an occupied parallelepiped
of diameter
larger than or equal to $i$ in $\Lambda_\beta$, i.e., 
$$\displaylines{
  \tau^i\,=\,\inf\,
  \left\{
  \,t\geq 0:
\begin{matrix}
\text{there is an occupied parallelepiped
  included in $\Lambda_\beta$ 
}\\
\text{ whose
$\dinf$ diameter
is larger than or equal to $i$}
\end{matrix}
\right\}
}$$
The restriction of the process
$(\s_{\L_\beta;t})_{t\geq 0}$
to the sites which are the neighbors of a face of an occupied parallelepiped
is a $d-1$ dimensional growth process whose rates satisfy the hypothesis
of our model. From the induction hypothesis,
we know that, up to a SES event, the $d-1$ dimensional process in a box
of sidelength $\exp(\beta K)$ is fully occupied at
a time
$$\exp\Big(\beta\big(
\max(\Gamma_{d-1}-(d-1)K,\kappa_{d-1})+\varepsilon\big)
\Big)\,.
$$
This implies that, up to a SES event, 
the box $\L_\beta$ is fully occupied
at time
$$\displaylines{ 
\tau^ {\exp(\beta L)}
\,\leq\,
\tau_{\text{nucleation}}
\,+\,
\sum_{1\leq i< 
{\exp(\beta L)}
}
(\tau^{i+1}
- \tau^{i})
\,\leq\,
\exp\Big(\beta\big(
\Gamma_{d}-dL+\varepsilon\big)
\Big)
\hfill\cr
+
\sum_{1\leq i< \exp(\beta L)
}
2d\exp\Big(\beta\big(
\max(\Gamma_{d-1}-\frac{d-1}{\beta}\ln i,\kappa_{d-1})+\varepsilon\big)
\Big)
}$$
We consider two cases.

\noindent
$\bullet$ First case: $L\leq L_{d-1}$.
Notice that $L_0=0$, hence this case can happen only whenever $d\geq 2$.
In this case, we have 
$$
\forall i< \exp(\beta L)\qquad
\kappa_{d-1}\,\leq\,
\Gamma_{d-1}-\frac{d-1}{\beta}\ln i$$
and 
$$\displaylines{ 
\sum_{1\leq i< \exp(\beta L)}
\exp\Big(\beta
\max(\Gamma_{d-1}-\frac{d-1}{\beta}\ln i,\kappa_{d-1})
\Big)
\hfill\cr
\,\leq\,
\exp(\beta
\Gamma_{d-1})
\sum_{1\leq i< \exp(\beta L)}
\frac{1}{i^{d-1}}
\cr
\,\leq\,
\exp(\beta
\Gamma_{d-1})
\sum_{1\leq i< \exp(\beta L)}
\frac{1}{i}
\,\leq\,
\beta L
\exp(\beta
\Gamma_{d-1})\,.
}$$
\noindent
$\bullet$ Second case: $L> L_{d-1}$. We have then
$$\displaylines{ 
\sum_{
\exp(\beta L_{d-1})
\leq i< \exp(\beta L)}
\exp\Big(\beta
\max(\Gamma_{d-1}-\frac{d-1}{\beta}\ln i,\kappa_{d-1})
\Big)
\hfill\cr
\,\leq\,
\Big(
\exp(\beta L)-
\exp(\beta L_{d-1})\Big)
\exp(\beta \kappa_{d-1})
\cr
\,\leq\,
\exp\Big(\beta (L+\kappa_{d-1})\Big)\,.
}$$
We conclude that, in both cases, for any 
$\varepsilon>0$, up to a SES event,
the box $\L_\beta$ is fully occupied
at a time
$$\displaylines{ 
2d\beta L\exp(\beta\varepsilon)
\left(
\exp\Big(\beta(
\Gamma_{d}-dL)
\Big)
+
\exp(\beta
\Gamma_{d-1})
+
\exp\Big(\beta (L+\kappa_{d-1})\Big)
\right)\,.
}$$
Therefore, for any $\kappa$ such that
$$\kappa\,>\,
\max\big(
\Gamma_{d}-dL,
\Gamma_{d-1},
L+\kappa_{d-1}\Big)$$
the probability that 
the box $\L_\beta$ is not fully occupied
at a time
$\exp(\beta\kappa)$
is SES.
If $L\leq L_d$ then
$$\max\big(
\Gamma_{d}-dL,
\Gamma_{d-1},
L+\kappa_{d-1}\Big)
\,=\,
\Gamma_{d}-dL
$$
and we have the desired estimate.
Suppose next that
$L> L_d$. By the previous result,
we know that,
for any $\kappa>\kappa_d$,
up to a SES event,
a box of sidelength
$\exp(\beta L_d)$
is fully occupied
at a time
$\exp(\beta\kappa)$.
We cover $\L_\beta$ by boxes
of sidelength
$\exp(\beta L_d)$. Such a cover contains at most
$\exp(\beta dL)$ boxes, thus
$$\displaylines{ 
\P\tonda{
\text{$\L_\beta$ is not fully occupied at time $\tau_\beta$}
}\hfill\cr
\,\leq\,
\P\left(
\begin{matrix}
\text{there exists a box included in $\L_\beta$
of sidelength 
}\\
\text{
$\exp(\beta L_d)$
which is not fully occupied at time $\tau_\beta$}
\end{matrix}
\right)\cr
\,\leq\,
\exp(\beta dL)\,
\P\tonda{
\begin{matrix}
\text{the box $\L(
\exp(\beta L_d))$ is not}\\
\text{fully occupied at time $\tau_\beta$}
\end{matrix}
}\,.
}$$
The last probability being SES, we are done.

%
\medskip

\noindent
{\bf Acknowledgements:} 
Rapha\"el Cerf thanks Roberto Schonmann for discussions
on this problem while he visited UCLA in 1995.

\bibliographystyle{alpha}
\bibliography{cema}
\end{document}